\documentclass[12pt]{amsart}
 \usepackage{amsmath}
 \usepackage{amssymb}
 \usepackage{amsthm}
\usepackage[utf8]{inputenc}
\pdfoutput=1
\usepackage[margin=1in]{geometry}
\usepackage{parskip}
\usepackage{verbatim}
\usepackage{xcolor}
\usepackage{tikz-cd}
\usepackage[pagebackref]{hyperref}
\usepackage[numbers]{natbib}
\usepackage{doi}
\usepackage{cleveref}

\newcommand{\mc}[1]{\mathcal{#1}}
\newcommand{\mf}[1]{\mathfrak{#1}}

\newcommand{\bb}[1]{\mathbb{#1}}
\newcommand{\olin}[1]{\overline{#1}}

\newcommand{\wt}[1]{\widetilde{#1}}

\newcommand{\N}{\mathbb{N}}

\newcommand{\Q}{\mathbb{Q}}

\newcommand{\inv}{^{-1}}

\newcommand{\demog}{\delta}

\let\hom\relax

\DeclareMathOperator{\hom}{Hom}
\DeclareMathOperator{\im}{im}

\DeclareMathOperator{\Spec}{Spec}

\DeclareMathOperator{\Min}{Min}
\DeclareMathOperator{\Ass}{Ass}

\DeclareMathOperator{\Ccore}{C}

\newtheorem{theorem}{Theorem}[section]
\newtheorem{thm}[theorem]{Theorem}
\newtheorem{lemma}[theorem]{Lemma}
\newtheorem{prop}[theorem]{Proposition}
\newtheorem{cor}[theorem]{Corollary}

\newtheorem*{thmA}{Theorem A}
\newtheorem*{thmB}{Theorem B}
\newtheorem*{thmC}{Theorem C}

\theoremstyle{definition}
\newtheorem{defn}[theorem]{Definition}
\newtheorem{rmk}[theorem]{Remark}
\newtheorem{zb}[theorem]{Example}
\newtheorem{notation}[theorem]{Notation}

\newtheorem*{assume}{Assumptions}
\newtheorem*{ackblock}{Acknowledgements}

\theoremstyle{remark}

\title{The Cartier core map for Cartier algebras}
\author{Anna Brosowsky}
\address{Department of Mathematics, University of Michigan, Ann Arbor MI 48109}
\email{annabro@umich.edu}
\thanks{Partially supported by NSF DMS grants \#1840234 and \#2101075.}


\begin{document}

\begin{abstract}
Let $R$ be a commutative Noetherian $F$-finite ring of prime characteristic and let $\mathcal{D}$ be a Cartier algebra. 
We define a self-map on the Frobenius split locus of the pair~$(R,\mathcal{D})$ by sending a point $P$ to the splitting prime of $(R_P, \mathcal{D}_P)$. We prove this map is continuous, containment preserving, and fixes the $\mathcal{D}$-compatible ideals.  
We show this map can be extended to arbitrary ideals $J$, where in the Frobenius split case it gives the largest $\mathcal{D}$-compatible ideal contained in~$J$.
Finally, we apply Glassbrenner's criterion to prove that the prime uniformly $F$-compatible ideals of a Stanley-Reisner rings are the sums of its minimal primes.
\end{abstract}

\maketitle

\section{Introduction}
\label{sec:introduction}

Frobenius splitting is an important tool in characteristic $p$ commutative algebra and algebraic geometry. Locally, Frobenius splitting (or $F$-purity) is a restriction on  singularities,  analogous to log canonicity for complex singularities. Strong $F$-regularity is a strengthening of Frobenius splitting, analogous to how Kawamata log terminality is a strengthening of log canonicity. There is much interest in understanding the world of Frobenius split objects that are not strongly $F$-regular.

For a local ring, Aberbach and Enescu introduced the \emph{splitting prime} as a way to measure the difference between Frobenius splitting and strong $F$-regularity---the elements in the splitting prime are obstructions to strong $F$-regularity, and in particular, the splitting prime of a domain is zero precisely for strongly $F$-regular rings \cite{Aberbach+Enescu.05}. Aberbach and Enescu's splitting prime can also be described as  the largest uniformly $F$-compatible ideal in the sense of Schwede \cite{Schwede.10a}. In this paper, we will be working with a generalization of the splitting prime following two different directions.

Frobenius splitting and strong $F$-regularity have  been generalized to further settings, including pairs $(X, \Delta)$ consisting of a $\Q$-divisor $\Delta$ on a smooth variety $X$ of characteristic $p$; pairs $(R,\mf a^t)$ where $R$ is a ring of characteristic $p$, $\mf a$ is an ideal, and the formal exponent $t$ is a positive real number; as well as more general settings \cite{Hara+Watanabe.02,Schwede.08,Schwede.10,Takagi.04}.
The Cartier algebra (see \Cref{full-cart-alg-def}) gives a unified and more general approach, allowing us to talk about Frobenius splitting and strong $F$-regularity for an arbitrary subalgebra of the Cartier algebra \cite{Schwede.11a}.
An important problem is to understand the extent to which strong $F$-regularity fails for an arbitrary Frobenius split subalgebra of the Cartier algebra \cite{Blickle+etal.12,Schwede.10a}.

Fix a Frobenius split pair~$(R,\mc D)$, where $\mc D$ is a Cartier subalgebra (see \Cref{cartier-mod-defs}). One main theme of this paper is to consider splitting primes via the perspective of the Cartier core map
\[
\Ccore_{\mc D}:\Spec R \to \Spec R
\]
which assigns to each prime $P\in \Spec R$ the splitting prime $\Ccore_{\mc D}(P)$ corresponding to the pair~$(R_P, \mc D_P)$. Alternatively, if $R$ is not Frobenius split one can define $\Ccore_{\mc D}$ on the Frobenius split locus, $\mc U_{\mc D}$, of $(R,\mc D)$. In this context, we show the following main result.

\begin{thmA}[\Cref{c-map-summary-thm}] 
Let $R$ be an $F$-finite Noetherian ring of characteristic $p$, and let $\mc D$ be a Cartier subalgebra. Then the Cartier core map
\[
\mathcal U_{\mc D} \rightarrow \Spec R \qquad\qquad   P\mapsto C_{\mathcal D}(P)
\]
is a continuous containment preserving map on the $F$-pure locus $\mathcal U_{\mc D}$ of the pair~$(R,\mc D)$ which fixes the $\mc D$-compatible ideals. The image of $C_{\mathcal D}$ is the set of prime $\mc D$-compatible ideals and is always finite. The image is the set of minimal primes of $R$ precisely when the pair~$(R,\mc D)$ is strongly $F$-regular.
\end{thmA}

In another direction, for an arbitrary (not necessarily Frobenius split) pair~$(R,\mc D)$, we can associate to \emph{any} ideal $J$ of $R$  the largest $\mc D$-compatible ideal  $\Ccore_{\mc D}(J)$ contained in $J$. We call this ideal the \emph{Cartier core} of $J$, following 
Badilla-C\'espedes, who considered this earlier for the non-pair setting  \cite{Badilla-Cespedes.21}. We will prove some basic facts about the Cartier core, many of which mirror results of Schwede in the triples setting and of Badilla-C\'espedes in the setting where $\mc D$ is the full Cartier algebra.

Returning to the non-pair setting, we also give the following explicit formula for the Cartier core of an arbitrary ideal in a quotient of a regular ring, making use of a criterion for strong $F$-regularity due to Glassbrenner \cite{Glassbrenner.96}.

\begin{thmB}[\Cref{c-ideal-quo-reg-ring}]
Let $S$ be a regular $F$-finite ring, let $I\subseteq J$ be ideals of $S$, and let $R=S/I$. 
Then
\[
C_R(J/I) = \left(\bigcap_{e>0} J^{[p^e]} :_S (I^{[p^e]} :_S I)\right)/I.
\]
\end{thmB}

This presentation of the Cartier core allows us to prove that the Cartier core map commutes with basic operations such as localizing, 
adjoining a variable, and in the case of quotients of polynomial rings, with homogenization (see \Cref{localization}, \Cref{qr-add-var}, \Cref{qr-homogenization}). As an application of these techniques, we give an exact description of the Cartier core map in the case of Stanley-Reisner rings.

\begin{thmC}[\Cref{stanley-reisner-c-map}, \Cref{stanley-reisner-general-c-map}]
Let $R$ be a Stanley-Reisner ring over a field that has prime characteristic and is $F$-finite. Let $Q$ be any prime ideal of $R$. Then
\[
C_R(Q) = \sum_{\substack{P \in \Min(R) \\ P\subseteq Q}} P.
\]
In particular, the image of the Cartier core map on primes, i.e., the set of generic points of $F$-pure centers of $R$, is the set of sums of minimal primes. Further, if $J$ is any ideal, then
\[
\Ccore_R(J) = \sum_{\substack{\mc Q\subset \Min(R) \\ \big({\bigcap\limits_{P\in \mc Q}} P\big) \subset J  }} 
    \left( \bigcap_{P\in \mc Q} P \right).
\]
\end{thmC}

This theorem extends existing work on computing certain specific uniformly $F$-compatible ideals and $\mc D$-compatible ideals, including the splitting prime and test ideals, for Stanley-Reisner rings \cite{Aberbach+Enescu.05,Badilla-Cespedes.21,Enescu+Ilioaea.20,Vassilev.98}.

\begin{assume}
All rings in this paper (other than the Cartier algebras) are commutative, Noetherian, and unital. Furthermore, all such rings are of prime characteristic $p$ and are $F$-finite.
\end{assume}

\begin{ackblock} 
I am grateful to my advisor, Karen Smith, for all of her guidance and suggestions. I would also like to thank Shelby Cox and Swaraj Pande for the helpful conversations. Thanks to W\'agner Badilla-C\'espedes, Karl Schwede, Kevin Tucker, Janet Vassilev, and the referee for their feedback on an earlier version of this paper; and especially to Anne Fayolle for pointing out an error.
\end{ackblock}

\section{Background}
\label{sec:background}
For a ring $R$ of prime characteristic $p$, the Frobenius endomorphism is the ring map $F:R\to R$ where $F(r) = r^p$. To distinguish the copies of $R$, 
we will write $F_*R$ for the codomain. As a ring, this Frobenius pushforward $F_*R = \{F_*r\: | \: r\in R\}$ is exactly the same as $R$, just with this formal symbol $F_*$ prepended everywhere. For example, multiplication is $(F_* r)(F_* s) = F_*(rs)$. The benefit of this notation is that it clarifies the $R$-module structure induced by $F$---the Frobenius map is now written as $F: R \to F_*R$ so that $F(r) = F_*(r^p)$, and the $R$-module action is now written $rF_*s = F_*(r^p s)$.
We can iterate the Frobenius, writing $F^e: R\to F_*^e R$, where $F^e(r) = F_*^e(r^{p^e})$ and $rF_*^e s = F_*^e (r^{p^e}s)$.

We will utilize this $R$-module structure on $F_*^eR$, but first we need a cohesive way to consider only certain maps in $\hom_R(F_*^e R, R)$. First, given any map $\psi \in \hom_R(F_*^dR, R)$, we write $F_*^e \psi: F_*^{e+d}R\to F_*^eR$ for the \emph{Frobenius pushforward} of the map, where 
\[
(F_*^e\psi)(F_*^{e+d} r) = (F_*^e \psi)(F_*^e(F_*^d r)) = F_*^e(\psi(F_*^d r)).
\]
Now we define a (non-commutative) multiplication on the abelian group $\bigoplus_e \hom_R(F_*^e R, R)$ as follows. Given maps $\phi \in \hom_R(F_*^e R, R)$ and $\psi \in \hom_R(F_*^d R, R)$, we define their product as
\begin{equation}
\label{eq:cartier-mult}
\phi\cdot \psi = \phi \circ F_*^e \psi.
\end{equation}
More concretely, for any $r\in R$ we have 
\[
(\phi\cdot \psi) (F_*^{e+d} r) = \phi \left( F_*^e(\psi(F_*^d r))\right).
\]

\begin{defn}
\label{full-cart-alg-def}
The \emph{(full) Cartier algebra} on $R$ is the graded non-commutative ring 
\[
\mc C_R = \bigoplus_{e\geq 0} \hom_R(F_*^e R, R),
\]
where multiplication is as defined in \Cref{eq:cartier-mult}.
\end{defn}
Note that we are writing $F_*^0R$ to mean $R$ as an $R$-module, so that $(\mc C_R)_0 = \hom_R(R,R)\cong R$. We will often write the ``multiplication by $c$'' map as simply $c$, and its pushforward as $F_*^ec$, so that $(F_*^ec)(F_*^e r) = F_*^e(cr)$.
However, this copy of $R$ is rarely central in $\mc C_R$, because for $\phi$ of degree $e$, we have $r\cdot \phi = \phi \cdot r^{p^e} $.
In particular, $R$ is central only if $R = \bb F_p$.

\begin{defn}
\label{cart-subalg-def}
A \emph{Cartier subalgebra} $\mc D$ is a graded subalgebra of $\mc C_R$ such that $\mc D_0 = R$. In particular, $\mc D$ has the form $\mc D = \bigoplus_e \mc D_e$ where $\mc D_e \ \subseteq \hom_R(F_*^eR, R)$ for all $e\geq 0$.
\end{defn}

\begin{zb}[{\cite[Rmk.~3.10]{Schwede.11a}}]
Let $(R,\mf a^t)$ be a pair where $\mf a$ is an ideal and the formal exponent $t$ is a positive real number. Then the corresponding Cartier subalgebra $\mc C^{\mf a^t}$ has 
\[
\mc C^{\mf a^t}_e = \hom_R(F_*^e R, R)\cdot \mf a^{\lceil t(p^e-1)\rceil}.
\]
\end{zb}

\begin{defn}
\label{cartier-mod-defs}
Let $R$ be a ring of prime characteristic, and let $\mc D$ be a Cartier subalgebra on $R$.
\begin{itemize}

\item The pair~$(R,\mc D)$ is \emph{$F$-finite} if $R$ is $F$-finite, i.e., $F_*R$ is a finite $R$-module. Every ring $R$ in this paper will be $F$-finite.

\item The pair~$(R,\mc D)$ is \emph{Frobenius split} or \emph{(sharply) $F$-pure} if there exists some  $e>0$ and some $\phi \in \mc D_e$ with $\phi(F_*^e 1)=1$.

\item If $c$ is an element of $R$, then the pair~$(R,\mc D)$ is \emph{eventually Frobenius split along $c$}, or \emph{$F$-pure along $c$} if there exists some $e>0$ and some $\phi \in \mc D_e$ with $\phi(F_*^e c)=1$.

\item The pair~$(R,\mc D)$ is \emph{strongly $F$-regular} if it is eventually Frobenius split along every $c$ which is not in any minimal prime of $R$.
\end{itemize}
\end{defn}
We will follow the example of Blickle, Schwede, and Tucker and omit the adjective ``sharp'' when discussing $F$-purity of pairs \cite[Def.~2.7]{Blickle+etal.12}.
Observe that if $\phi \in \mc D_e$ is a splitting of $F^e$, then there is a splitting in any multiple of the degree, given by $\phi^n \in \mc D_{en}$.

Since we will consider only pairs $(R,\mc D)$ where $R$ is Noetherian and $F$-finite, this means that for any ring $S$ such that $R\to S$ is flat, we have by \cite[Thm.~7.11]{Matsumura.89},
\[
S\otimes_R \hom_R(F_*^eR, R)\cong \hom_S(S\otimes_R F_*^e R, S).
\]
In the case that $S$ is a localization of $R$ we further know that $S$ commutes with the Frobenius, that is, for any multiplicative set $W$,
\begin{align*}
W\inv R \otimes_R \hom_R(F_*^eR, R) &\cong \hom_{W\inv R}(F_*^e (W\inv R), W\inv R).
\end{align*}
We will use this isomorphism freely: if $\frac{r}{w} \otimes \phi$ is a pure tensor in $W\inv R \otimes \hom_{R}(F_*^e R, R)$, we will identify this with the map in $\hom_{W\inv R}(F_*^e (W\inv R), W\inv R)$ which sends 
$F_*^e(\frac{s}{u})$ to $\frac{r\phi(F_*^e( su^{p^e-1}))}{wu}$. 
This identification is easier to understand if we first rewrite 
$F_*^e(\frac{s}{u})$ as 
\[
F_*^e\left(\frac{su^{p^e-1}}{u^{p^e}}\right) 
= \frac{1}{u}\cdot \frac{F_*^e(su^{p^e-1})}{1}
= \frac{F_*^e(su^{p^e-1})}{u}.
\]

Thus we have a natural containment $W\inv R\otimes_R \mc D_e\subseteq  (\mc C_{W\inv R})_e$. We can therefore construct a new Cartier subalgebra $W\inv \mc D$ on $W\inv R$ using this isomorphism, so that
\begin{align*}
(W\inv \mc D)_e = W\inv R\otimes \mc D_e.
\end{align*}
When we are localizing at a prime ideal $P$, we write this Cartier subalgebra as $\mc D_P$.

Now that we have the setup to discuss localizations of Cartier subalgebras, we can state and prove the following standard result on the Frobenius split locus in the setting of Cartier algebra pairs.
\begin{thm}
\label{f-pure-open-local}
Let $R$ be an $F$-finite ring, and $\mc D$ a Cartier subalgebra. Then the set of primes $P$ of $R$ at which $(R_P, \mc D_P)$ is $F$-pure is open. Further, the pair~$(R,\mc D)$ is $F$-pure if and only if the localized pair~$(R_P, \mc D_P)$ is $F$-pure for all primes $P$.
\end{thm}
\begin{proof}
For any $e$, we get a module map $\Psi_e: \mc D_e \to R$ via evaluation at~$F_*^e 1$. 
The pair~$(R,\mc D)$ is $F$-pure exactly when this map is surjective for some $e>0$, or equivalently, when there exists an $e>0$ such that $R/\im \Psi_e=0$. The localization $(\Psi_e)_P$ corresponds to the evaluation map $(\mc D_P)_e \to R_P$, so the pair~$(R_P, \mc D_P)$ is \emph{not} $F$-pure if and only if $R_P/\im (\Psi_e)_P \neq 0$ for all $e$.
Thus the non-$F$-pure locus is precisely the closed set $\bigcap_{e>0} \bb V(\im \Psi_e)$.

For the second statement, if $(R,\mc D)$ is $F$-pure, then there exists some $e>0$ and $\phi\in \mc D_e$ with $\phi(F_*^e 1)=1$. By definition, the localization $\phi_P : F_*^e(R_P) \to R_P$ is in $(\mc D_P)_e$, and so $(R_P, \mc D_P)$ is also $F$-pure.

Conversely, if each $(R_P, \mc D_P)$ is $F$-pure, then the complements of the sets $\bb V(\im \Psi_e)$ give an open cover of $\Spec R$. Since $\Spec R$ is compact, only finitely many are needed, say, the complements of $\bb V(\im \Psi_{e_1}), \ldots, \bb V(\im \Psi_{e_t})$. Then taking $e=e_1\cdots e_t$ to be the product of these indices, we must have that $(R_P, \mc D_P)$ has a splitting in $\mc D_e$ for every prime $P$. Thus the map $\Psi_e$ is surjective at every prime, and therefore is surjective.
\end{proof}
This proof in fact shows that for any $c$, the set of primes $P$ such that $(R_P, \mc D_P)$ is not eventually Frobenius split along $c$ is closed. Further, it also shows that $(R,\mc D)$ is eventually Frobenius split along $c$ if and only if $(R_P, \mc D_P)$ is for every prime ideal $P$. In particular, this shows that just like in the classical case, $(R,\mc D)$ is strongly $F$-regular if and only if every $(R_P, \mc D_P)$ is as well.

\section{The Cartier core map}

Fix a pair~$(R,\mc D)$ where $R$ is an $F$-finite and Frobenius split ring and  where $\mc D$ is a Cartier subalgebra. In this section we will define an explicit continuous map 
\[
\Ccore_{\mc D}: \Spec R \to \Spec R
\]
that has some especially nice properties. The image of our map is the set of $\mc D$-compatible primes of $\Spec R$, which in the case $\mc D = \mc C_R$ is the set of (generic points of) $F$-pure centers. 
If $R$ is not Frobenius split, we can instead define $\Ccore_{\mc D}$ on the open locus of Frobenius split points.
More generally, the map $\Ccore_{\mc D}$ can be viewed as an endomorphism defined on the set of \emph{all ideals} of $R$ (not necessarily proper), which is especially interesting on the class of radical ideals in a Frobenius split ring.

\begin{defn}
Let \(R\) be an $F$-finite ring of prime characteristic. Let \(J\) be an ideal of \(R\). Let $\mc D \subseteq \mc C_R$ be a Cartier subalgebra.
Then the \emph{Cartier core} of \(J\) in \(R\) with respect to $\mc D$ is 
\[
\Ccore_{\mc D}(J) = \left\{r\in R \: | \: \phi(F_*^e r)\in J \ \  \forall e>0,\ \forall\phi\in \mc D_e\right\}.
\]
\end{defn}

We will write $\Ccore_R(J)$ to mean the Cartier core with respect to the full Cartier algebra $\mc C_R$, and just $\Ccore(J)$ when the ring and Cartier subaglebra are clear from context. In the case that $\mc D = \mc C_R$, the Cartier core $\Ccore_R(J)$ is also denoted (e.g., in \cite{Badilla-Cespedes.21}) as $\mc P(J)$.

\begin{notation}
The \emph{\(e\)-th Cartier contraction} of $J$ with respect to $\mc D$ is 
\[
A_{\mc D_e}(J) = \left\{ r \in R \: | \: \phi(F_*^e r) \in J\ \  \forall \phi \in \mc D_e\right\}.
\]
\end{notation}
We can express the Cartier core in terms of the Cartier contractions as
\[
\Ccore_{\mc D}(J) = \bigcap_{e>0} A_{\mc D_e}(J).
\]
We can also express the Frobenius pushforward of the $e$-th Cartier contraction as
\[
F_*^e(A_{\mc D_e}(J)) = \bigcap_{\phi \in \mc D_e}\phi\inv(J).
\]
When $\mc D = \mc C_R$, the \(e\)-th Cartier contraction \(A_{\mc D_e}(J)\) is sometimes denoted by $J_e$.

Note that for an $F$-finite pair~$(R,\mc D)$, $A_{\mc D_e}(J)$ and \(\Ccore_{\mc D}(J)\) are ideals.
Both are clearly additively closed, so it suffices to check that if $a\in A_{\mc D_e}(J)$ and $r\in R$, then $ra\in A_{\mc D_e}(J)$.
For any $\phi \in \mc D_e$, we have $\phi(F_*^e(ra)) = (\phi \cdot r)(F_*^e a)$, which is in $J$ since $a\in A_{\mc D_e}(J)$.

The Cartier core was defined for the case $\mc C_R = \mc D$ by Badilla-C\'espedes \cite[Def.~4.12]{Badilla-Cespedes.21} as a generalization of Aberbach and Enescu's splitting prime \cite{Aberbach+Enescu.05} and of Brenner, Jeffries, and N\'u\~{n}ez Betancourt's differential core \cite{Brenner+etal.19}. Here we generalize this definition to the context of pairs, similar to Blickle, Schwede, and Tucker's generalization of the splitting prime to the context of pairs \cite{Blickle+etal.12}.

To motivate the definition of the Cartier core, note that the condition \(J\subseteq \Ccore_{\mc D}(J)\), i.e., that \(\phi(F_*^e(J))\subseteq J\) for all $e$ and for all $\phi\in \mc D_e$, is precisely the condition that \(J\) is $\mc D$-compatible. 
In the case where $\mc D$ is the full Cartier algebra, this is the same as saying $J$ is uniformly $F$-compatible. 
In fact, it is known that  when $R$ is $F$-pure, $\Ccore_{R}(J)$ is the largest uniformly $F$-compatible ideal contained in \(J\) \cite[Prop.~4.11]{Badilla-Cespedes.21}. We will see in \Cref{largest-D-compat} that when the pair~$(R,\mc D)$ is Frobenius split, $\Ccore_{\mc D}(J)$ is the largest $\mc D$-compatible ideal contained in $J$.

Further, as the next two results show, the Cartier core of a prime ideal $P$ carries information about the localization $(R_P, \mc D_P)$.

\begin{prop}[{\cite[Prop.~2.12]{Blickle+etal.12}}]
\label{f-split-proper}
Let $(R,\mc D)$ be an $F$-finite pair and let $P$ be a prime ideal of $R$. 
Then $r\notin \Ccore_{\mc D}(P)$ if and only if the pair~$(R_P, \mc D_P)$ is $F$-pure along $r/1$. In particular, $(R_P, \mc D_P)$ is $F$-pure if and only if $\Ccore_{\mc D}(P)$ is proper.
\end{prop}
\begin{proof}
Since $\mc D_P = \mc D\otimes R_P$, saying $\phi(F_*^e(r))\in P$ for some $\phi\in\mc D_e \subseteq \hom_R(F_*^e R, R)$ is equivalent to saying $\phi(F_*^e(r/1))\in PR_P$, viewing $\phi\in \mc (D_P)_e\subseteq \hom_{R_P}(F_*^e(R_P), R_P)$.

The pair~$(R_P,\mc D_P)$ is $F$-pure if and only if there is some $\phi\in \mc D_P$ such that $\phi(F_*^e(1))$ is a unit, i.e., not in $PR_P$, which by the above is equivalent to having $1\notin \Ccore_{\mc D}(P)$.
\end{proof}

\begin{prop}[{Cf. \cite[Thm.~2.11,Prop.~2.12]{Blickle+etal.12}}]
\label{str-f-reg=in-minl-prime}
Let $(R,\mc D)$ be an $F$-finite pair and let $P$ be a prime ideal of $R$. Then the pair~$(R_P, \mc D_P)$ is strongly $F$-regular if and only if $\Ccore_{\mc D}(P)$ is contained in some minimal prime of $R$. 
\end{prop}
\begin{proof}
The pair~$(R_P, \mc D_P)$ is strongly $F$-regular if and only if $(R_P, \mc D_P)$ is $F$-pure along every non-zero divisor, i.e., $\Ccore_{\mc D}(P)$ is contained in the union of the minimal primes of $R$. Since $\Ccore_{\mc D}(P)$ is an ideal, prime avoidance says this is equivalent to having $\Ccore_{\mc D}(P)$ contained in some minimal prime of $R$. 
\end{proof}

Now that we have provided some motivation for the Cartier core construction, we will discuss some of its nice properties.

\begin{prop}
\label{containment}
Let $(R,\mc D)$ be an $F$-finite pair. If $J_1 \subseteq J_2$ in \(R\), then $\Ccore_{\mc D}(J_1)\subseteq \Ccore_{\mc D}(J_2)$.
\end{prop}
\begin{proof}
For every $e$, $A_{\mc D_e}( J_1)\subseteq A_{\mc D_e}(J_2)$, since if $\phi(F_*^e r)\in J_1$ for some $\phi \in \mc D_e$, we also have $\phi(F_*^e r)\in J_2$. Taking the intersection over all $e$ gives our result.
\end{proof}

\begin{prop}[{Cf. \cite[Prop~4.6]{Badilla-Cespedes.21}}]
\label{arbitrary-intersection-c-ideal}
Let $\{J_\alpha\}$ be an arbitrary collection of ideals in an $F$-finite ring $R$, and let $\mc D$ be a Cartier subalgebra. Then 
\[
\Ccore_{\mc D}\left(\bigcap_\alpha J_\alpha\right) = \bigcap_\alpha \Ccore_{\mc D}(J_\alpha). 
\]
\end{prop}
\begin{proof}
We see that
\begin{align*}
\Ccore_{\mc D}\left(\bigcap_\alpha J_\alpha\right) &= \left\{r\in R\: | \: \phi(F_*^e r) \in \bigcap_\alpha J_\alpha \; \forall e, \forall \phi\in \mc D_e \right\} \\
    &= \bigcap_\alpha \left\{r\in R\: | \: \phi(F_*^e r) \in J_\alpha \ \forall e, \forall \phi\in \mc D_e \right\} \\
    &= \bigcap_\alpha \Ccore_{\mc D}(J_\alpha) \qedhere
\end{align*}
\end{proof}
In particular, the set of Cartier cores with respect to $\mc D$ is closed under arbitrary intersection. We will see in \Cref{prop:arbitrary-sum-c-ideals} that this set is also closed under arbitrary sum for $F$-pure pairs.

Our next goal is to show that the Cartier core construction commutes with localization. To do so, we need the following lemma.
\begin{lemma}
\label{localization}
Let $(R,\mc D)$ be an $F$-finite pair, let $Q$ be a $P$-primary ideal of $R$, and let $W$ be a multiplicative set avoiding $P$, so that $W\cap P=\emptyset$. Then
\[
\Ccore_{W\inv \mc D}(QW\inv R) \cap R = \Ccore_{\mc D}(Q).
\]
\end{lemma}
\begin{proof}
By our discussion in \Cref{sec:background}, $W\inv \mc D_e$ is generated by the maps $\frac{\phi}{w}: F_*(W\inv R)\to R$ for $\phi\in \mc D_e$ and $w\in W$, where $\frac{\phi}{w}(F_*(\frac{s}{u})) = \frac{\phi(F_*^e(su^{p^e-1}))}{wu}$.
We will start by showing that $\frac{s}{1}\in A_{W\inv \mc D_e}(QW\inv R)$ if and only if $s\in A_{\mc D_e}(Q)$.

By definition, $\frac{s}{1}\in A_{W\inv \mc D_e}(QW\inv R)$ if and only if $\psi(F_*^e(\frac{s}{1}))\in QW\inv R$ for all
\(
\psi \in W\inv \mc D_e.
\)
This is equivalent to having
\[
\frac{\phi(F_*^e(s))}{w}\in QW\inv R
\]
for all $\phi \in \mc D_e$ and all $w\in W$. 
This means that we can write $\frac{\phi(F_*^e(s))}{w} = \frac{j}{u}$ for some $j\in Q$, $u\in W$, i.e., there exists $v\in W$ such that $vu \phi(F_*^e(s)) = vwj$.
The latter is in $Q$, but $vu\notin P$, so by $P$-primaryness of $Q$ we must then have $\phi(F_*^e s) \in Q$. This holds for all $\phi$ exactly when $s \in A_{\mc D_e}(Q)$.

Now we have shown our first claim, which implies $A_{\mc D}(Q) = A_{W\inv \mc D_e}(Q) \cap R$. Intersecting both sides over all $e>0$, we see
\[
\Ccore_{W\inv \mc D}(Q)\cap R = \Ccore_{\mc D}(Q).\qedhere
\]
\end{proof}

\begin{thm}
\label{thm:localization}
Let $(R,\mc D)$ be an $F$-finite pair, let $J$ be an ideal of $R$, and let $W$ be a multiplicative set avoiding every prime in $\Ass(J)$. Then
\[
\Ccore_{W\inv \mc D}(JW\inv R)\cap R = \Ccore_{\mc D}(J)
\quad \text{ and } \quad
\Ccore_{\mc D}(J)W\inv R = \Ccore_{W\inv \mc D}(JW\inv R).
\]
\end{thm}
\begin{proof}
Write $J= Q_1 \cap \cdots \cap Q_t$ a minimal primary decomposition of $J$ with corresponding primes $P_i = \sqrt {Q_i}$. Then since intersection commutes with applying $\Ccore_{\mc D}$ and with contraction,
\begin{align*}
\Ccore_{W\inv \mc D}(J)\cap R 
    &= \bigcap_{i=1}^t (\Ccore_{W\inv \mc D}(Q_i)\cap R).
\end{align*}
By \Cref{localization}, since $W\cap P_i=\emptyset$ we have $\Ccore_{W\inv \mc D}(Q_i)\cap R = \Ccore_{\mc D}(Q_i)$ and so
\[
\Ccore_{W\inv \mc D}(J)\cap R 
= \bigcap_{i=1}^t \Ccore_{\mc D}(Q_i) 
= \Ccore_{\mc D}(J).
\]

For the second equality, we note
\[
\Ccore_{W\inv \mc D}(J)
= (\Ccore_{W\inv \mc D}(J)\cap R)W\inv R 
= \Ccore_{\mc D}(J)W\inv R 
\]
since contracting then extending to a localization preserves ideals.
\end{proof}

Now that we have established the preliminary results for arbitrary ideals, we move to considering prime ideals. Our main results of the rest of this section can be summarized in the following theorem.
\begin{thm}
\label{c-map-summary-thm}
Let $R$ be an $F$-finite Noetherian ring, and let $\mc D$ be a Cartier subalgebra. Then the Cartier core construction with respect to $\mc D$ induces a well-defined, continuous, and containment preserving map on the $F$-pure locus of the pair~$(R,\mc D)$ which fixes $\mc D$-compatible ideals. The image of the map is the set of $\mc D$-compatible ideals in $\mc U_{\mc D}$ and is always finite. The image is the set of minimal primes of $R$ precisely when the pair~$(R,\mc D)$ is strongly $F$-regular.
\end{thm}

\begin{proof}
We have already seen in \Cref{containment} that the Cartier core is containment preserving, even without restricting to primes.
\Cref{c-map-well-defined} will show that the map $\Ccore:\mc U_{\mc D} \to \mc U_{\mc D}$ is well-defined. \Cref{cts-map} will show that this map is continuous, and \Cref{fin-many-C-cores} discusses the finiteness of the image. \Cref{cartier-core=D-compatible} will show that the image is precisely the set of $F$-pure $\mc D$-compatible ideals, which combined with \Cref{prop:c-ideal-stable} shows that all the $\mc D$-compatible ideals in $\mc U_{\mc D}$ are fixed.

The one statement that doesn't have a stand-alone proof elsewhere is the last one. $(R,\mc D)$ is strongly $F$-regular if and only if each $(R_P, \mc D_P)$ is strongly $F$-regular. By \Cref{str-f-reg=in-minl-prime}, this occurs exactly when each $\Ccore_{\mc D}(P)$ is contained in a minimal prime of $R$. But since $\Ccore_{\mc D}(P)$ is prime, this is equivalent to having $\Ccore_{\mc D}(P)$ be a minimal prime.
\end{proof}

It is known that the splitting prime, which in our notation is $\Ccore_{R}(\mf m)$ for $(R,\mf m)$ local, is indeed prime \cite[Thm.~3.3]{Aberbach+Enescu.05}, even in the case of an arbitrary Cartier subalgebra \cite[Prop.~2.12]{Blickle+etal.12}. After localizing, the same proof works here, which we repeat for the reader's convenience. 
\begin{prop}[{\cite[Prop.~2.12]{Blickle+etal.12}}]
\label{proper-prime}
If \(P\) is prime and \(\Ccore_{\mc D}(P)\) is proper, then \(\Ccore_{\mc D}(P)\) is prime.
\end{prop}
\begin{proof}
Suppose $c_0,c_1\notin \Ccore_{\mc D}(P)$. 
Then we will show $c_0c_1\notin \Ccore_{\mc D}(P)$. 
Our assumption means that $(R_P, \mc D_{P})$ is $F$-pure along each $c_i$, i.e., there exists an $e_i$ and $\psi_i\in \mc (D_P)_{e_i}$ such that $\psi_i(F_*^{e_i}c_i) = 1$. Then applying the map $\psi_1\circ F_*^{e_1}\psi_0 \circ F_*^{e_0+e_1}(c_1^{p^{e_0}-1})$ to $F_*^{e_0+e_1}(c_0c_1)$, where we are writing $F_*^{e_0+e_1}(c_1^{p^{e_0}-1})$ to mean multiplication by this ring element, we get
\begin{center}
\begin{tikzcd}[column sep=large]
F_*^{e_0+e_1} (R_P) \rar["F_*^{e_0+e_1}(c_1^{p^{e_0}-1})"] 
    & F_*^{e_0+e_1}(R_P) \rar["F_*^{e_1}\psi_0"]  
    & F_*^{e_1}(R_P) \rar["\psi_1"] 
    & R_P \\
F_*^{e_0+e_1}(c_0c_1) \rar[mapsto] 
    & F_*^{e_0+e_1}(c_0c_1^{p^{e_0}}) = F_*^{e_1}(c_1F_*^{e_0}(c_0)) \rar[mapsto] 
    & F_*^{e_1}c_1 \rar[mapsto] 
    & 1.
\end{tikzcd}
\end{center}
Rewriting this map as $\psi_1 \circ F_*^{e_1}\psi_0 \circ F_*^{e_0+e_1}(c_1^{p^{e_0}-1}) = \psi_1 \cdot \psi_0 \cdot c_1^{p^{e_0}-1}$, we see that it is in $(\mc D_P)_{e_0+e_1}$, and thus that that $(R_P, \mc D_P)$ is also $F$-pure along $c_0c_1$, as desired.
\end{proof}

\begin{prop}
\label{prop:c-ideal-in-primary-original-ideal}
Let $R$ be a characteristic $p$, $F$-finite ring, and let $\mc D$ be a Cartier subalgebra. If $Q$ is a $P$-primary ideal of $R$ and $\Ccore_{\mc D}(P)$ is proper, then $\Ccore_{\mc D}(Q)\subseteq Q$. 
\end{prop}
\begin{proof}
Since $\Ccore_{\mc D}(P)$ is proper, there is some $e>0$ and $\psi \in \mc D_e$ with $\psi(F_*^e 1)\notin P$.
Consider $r\notin Q$ and the map $\psi \circ (F_*^e (r^{p^e-1}))= \psi \cdot r^{p^e-1}$ in $\mc D_e$. Then by $P$-primaryness, 
\[
r\psi(F_* 1) = \psi(F_*^e r^{p^e}) = (\psi\cdot r^{p^e-1})(F_*^e r)\notin Q,
\]
and so $r\notin \Ccore_{\mc D}(Q)$ as desired.
%
\end{proof}

\begin{cor}
\label{c-map-well-defined}
Let $(R, \mc D)$ be an $F$-finite pair, with $F$-pure locus $\mc U_{\mc D}$. Then the Cartier core construction induces a well-defined map $\Ccore_{\mc D}:\mc U_{\mc D}\to \mc U_{\mc D}$. 
\end{cor}
\begin{proof}
Let $P$ be a prime ideal in $\mc U_{\mc D}$. Then $(R_P,\mc D_P)$ is Frobenius split, so \Cref{f-split-proper} gives that $\Ccore_{\mc D}$ is proper, and thus prime by \Cref{proper-prime}. This gives a map $\Ccore_{\mc D}:\mc U_{\mc D} \to \Spec R$.

Then \Cref{prop:c-ideal-in-primary-original-ideal} says $\Ccore_{\mc D}(P) \subseteq P$. Since the $F$-pure locus is open, this means $\Ccore_{\mc D}(P)$ must also be in the $F$-pure locus.
\end{proof}

\begin{cor}[{Cf. \cite[Cor.~4.8]{Schwede.10a}}]
\label{c-minl-prime}
Suppose the pair~$(R, \mc D)$ is $F$-finite and $F$-pure. If $P$ is a minimal prime of $R$, then $\Ccore_{\mc D}(P) = P$.
\end{cor}
\begin{proof}
Since $(R,\mc D)$ is $F$-pure, $\Ccore_{\mc D}(P)\subseteq P$. Since $\Ccore_{\mc D}(P)$ is prime by \Cref{proper-prime} and $P$ is minimal, we must have that $\Ccore_{\mc D}(P)=P$.
\end{proof}

\begin{cor}[{Cf. \cite[Prop.~4.5]{Badilla-Cespedes.21}}]
\label{c-ideal-in-original-ideal}
If the pair~$(R,\mc D)$ is $F$-finite and $F$-pure, then for any ideal $J$ we have $\Ccore_{\mc D}(J)\subseteq J$. 
\end{cor}
\begin{proof}
Write $J = Q_1\cap \cdots \cap Q_t$, where the $Q_i$ give a primary decomposition of $J$. Then by \Cref{arbitrary-intersection-c-ideal},
\[
\Ccore_{\mc D}(J) = \Ccore_{\mc D}(Q_1)\cap \cdots \cap \Ccore_{\mc D}(Q_t).
\]
Since $(R,\mc D)$ is Frobenius split, for every prime $P$ the pair~$(R_P, \mc D_P)$ is also Frobenius split, and thus has $\Ccore_{\mc D}(P)$ proper by \Cref{f-split-proper}. By \Cref{prop:c-ideal-in-primary-original-ideal}, each $\Ccore_{\mc D}(Q_i)\subseteq Q_i$. Intersecting, we get that $\Ccore_{\mc D}(J)\subseteq J$ as desired.
\end{proof}

\begin{prop}[{Cf. \cite[Lemma~3.5]{Schwede.10a}}]
\label{prop:arbitrary-sum-c-ideals}
Let $(R,\mc D)$ be an $F$-finite, $F$-pure pair, and let $\{J_\alpha\}_{\alpha \in \mc A}$ be a collection of ideals with $\Ccore_{\mc D}(J_\alpha)=J_\alpha$ for all $\alpha\in \mc A$. Then we have
\begin{align*}
\Ccore_{\mc D}\left(\sum_\alpha J_\alpha\right) &= \sum_\alpha \Ccore_{\mc D}(J_\alpha)
\end{align*}
\end{prop}
\begin{proof}
Since $J_\beta\subseteq \sum J_\alpha$, we have $\Ccore_{\mc D}(J_\beta)\subseteq \Ccore_{\mc D}\left(\sum_\alpha J_\alpha\right)$ for all $\beta\in \mc A$ by \Cref{containment}, and so 
\[
\sum_\alpha \Ccore_{\mc D}(J_\alpha)\subseteq \Ccore_{\mc D}\left(\sum_\alpha J_\alpha\right).
\]

For the reverse containment, we use our assumption that $\Ccore_{\mc D}(J_\alpha)=J_\alpha$ and \Cref{c-ideal-in-original-ideal} to see that
\[
\Ccore_{\mc D}\left(\sum J_\alpha\right)  
=  \Ccore_{\mc D}\left(\sum \Ccore_{\mc D}(J_\alpha)\right) 
\subseteq \sum \Ccore_{\mc D}(J_\alpha)
\]
which is our desired opposite inclusion.
\end{proof}

\begin{prop}
\label{prop:c-ideal-stable}
If the pair~$(R,\mc D)$ is $F$-finite and $F$-pure, then for any ideal $J$ in $R$,
\[
\Ccore_{\mc D}(J) = \Ccore_{\mc D}\left(\Ccore_{\mc D}(J)\right).
\]
\end{prop}
\begin{proof}
By \Cref{c-ideal-in-original-ideal}, we know that $\Ccore_{\mc D}(J)\subseteq J$.
Then $\Ccore_{\mc D}\left(\Ccore_{\mc D}(J)\right)\subseteq \Ccore_{\mc D}(J)$ by \Cref{containment}, so it suffices to show the other direction. 

Consider $f\notin \Ccore_{\mc D}\left(\Ccore_{\mc D}(J)\right)$. Thus there exists $e>0$ and $\phi\in \mc D_e$ with $\phi(F_*^e f)\notin \Ccore_{\mc D}(J)$. Then there must also exist $e'$ and $\phi'\in \mc D_{e'}$ with $\phi'(F_*^{e'} \phi(F_*^e(f))\notin J$. This term can be rewritten as $(\phi'\cdot \phi)(F_*^{e'+e}(f)) = \phi'\left(F_*^{e'}(\phi(F_*^e f))\right)$, and so $f\notin \Ccore_{\mc D}(J)$.
\end{proof}

\begin{rmk}
If the pair $(R,\mc D)$ is $F$-finite and $F$-pure,  then combining \Cref{c-ideal-in-original-ideal}, \Cref{containment}, and \Cref{prop:c-ideal-stable}, shows that $\Ccore_{\mc D}$ is a relative interior operation on ideals of $R$, in the sense of Epstein, R.G., and Vassilev \cite[Def.~2.2]{Epstein+etal.21}. 
\end{rmk}

The following result is known  when $\mc D = \mc C_R$ \cite{Badilla-Cespedes.21}, and  for triples $(R,\Delta, \mf a^t)$ \cite{Schwede.10a}. The proof in the Cartier algebra setting proceeds the same as Badilla-C\'espedes' proof, with a little care needed for the exponents used.
\begin{prop}[{Cf. \cite[Rmk.~4.14]{Badilla-Cespedes.21}, \cite[Cor.~3.3]{Schwede.10a}}]
\label{c-ideal-radical}
If the pair~$(R,\mc D)$ is $F$-finite and $F$-pure, then for any ideal $J$, the Cartier core $\Ccore_{\mc D}(J)$ is radical.
\end{prop}
\begin{proof}
Suppose $r\in \sqrt{\Ccore_{\mc D}(J)}$. Then there exists some $n$ so that $r^{p^n}\in \Ccore_{\mc D}(J)$. Since the pair is $F$-pure, there also exists some $\psi \in \mc D_d$ so that $\psi(F_*^d1)=1$. Take $e = nd$,
so that there is $\phi \in \mc D_e$ with $\phi(F_*^e 1)=1$, and so that \Cref{prop:c-ideal-stable} gives $r^{p^e}\in \Ccore_{\mc D}(J) = \Ccore_{\mc D}(\Ccore_{\mc D}(J))$. Then
\[
\phi(F_*^e(r^{p^e})) = r\phi(F_*^e 1) = r \in \Ccore_{\mc D}(J).\qedhere
\]
\end{proof}

The hypothesis that $(R,\mc D)$ be $F$-pure is necessary. Consider $R=k[x]/\langle x^2\rangle$ where $k$ is an $F$-finite field, and let $\mc D =\mc C_R$. This ring $R$ is non-reduced, so can't be $F$-pure. For any ideal $J\subset k[x]$, use $\olin J$ to denote the image of $J$ in $R$. Now using the presentation from \Cref{c-ideal-quo-reg-ring}, we compute
\begin{align*}
A_e(\olin{\langle x^2\rangle}) &= \olin{\langle x^2 \rangle^{[p^e]}:_{k[x]}\left(\langle x^2\rangle ^{[p^e]}:_{k[x]}\langle x^2 \rangle\right) }
= \olin{ \langle x^{2p^e} \rangle :_{k[x]} \langle x^{2p^e-2}\rangle } 
= \olin{\langle x^{2}\rangle}.
\end{align*}
Intersecting over all $e$, we see that $\Ccore_R(\olin{\langle x^2\rangle}) =\olin{\langle x^2\rangle}$, a non-radical ideal.

\begin{thm}[{Cf. \cite[Prop.~4.9, Thm.~4.10]{Badilla-Cespedes.21}}]
\label{cartier-core=D-compatible}
If the pair~$(R,\mc D)$ is $F$-finite and $F$-pure, then the set of Cartier cores with respect to $\mc D$, i.e., the set \(\left\{ \Ccore_{\mc D}(J) \: | \: J\text{ an ideal of \(R\)} \right\}\), is precisely the set of $\mc D$-compatible ideals.
\end{thm}
\begin{proof}
An ideal $J$ is $\mc D$-compatible precisely if $\phi(F_*^e(J))\subseteq J$ for all $e$ and for all $\phi \in \mc D_e$, and thus by construction $J$ is $\mc D$-compatible if and only if $J\subseteq \Ccore_{\mc D}(J)$. 
By \Cref{c-ideal-in-original-ideal}, if the pair~$(R,\mc D)$ is $F$-pure then this is equivalent to having $J=\Ccore_{\mc D}(J)$. This shows that every $\mc D$-compatible ideal is a Cartier core.

Conversely, the Cartier core $\Ccore_{\mc D}(J)$ is $\mc D$-compatible since by \Cref{prop:c-ideal-stable} we have $\Ccore_{\mc D}(J) = \Ccore_{\mc D}(\Ccore_{\mc D}(J))$.
\end{proof}

\begin{cor}[{Cf. \cite[Prop.~4.11]{Badilla-Cespedes.21}}]
\label{largest-D-compat}
If the pair~$(R,\mc D)$ is $F$-finite and $F$-pure and $J$ is an ideal of $R$, then $\Ccore_{\mc D}(J)$ is the largest $\mc D$-compatible ideal contained in $J$.
\end{cor}
\begin{proof}
$\Ccore_{\mc D}(J)$ is $\mc D$-compatible by the previous result. If another $\mc D$-compatible ideal~$J'$ has $\Ccore_{\mc D}(J) \subseteq J' \subseteq J$, then by \Cref{containment} we have $\Ccore_{\mc D}(\Ccore_{\mc D}(J)) \subseteq \Ccore_{\mc D}(J') \subseteq \Ccore_{\mc D}(J)$, and by \Cref{prop:c-ideal-stable} we in fact have $\Ccore_{\mc D}(J')=J'=\Ccore_{\mc D}(J)$.
\end{proof}

The following result, originally due to Schwede \cite[Cor.~5.10]{Schwede.09} and to Kumar and Mehta \cite[Thm.~1.1]{Kumar+Mehta.09}, captures another nice property of the Cartier core map. Recent work of Datta and Tucker \cite[Prop.~3.4.1]{Datta+Tucker.21} provides an alternate proof that uses similar language to the rest of this paper. 
\begin{prop}[{\cite[Prop.~3.4.1]{Datta+Tucker.21}}]
\label{fin-many-C-cores}
If $(R,\mc D)$ is an $F$-finite, $F$-pure pair, then there are only finitely many Cartier cores with respect to $\mc D$, i.e., there are only finitely many $\mc D$-compatible ideals.
\end{prop}


\begin{rmk}
If additionally $R$ is local, one can in fact get concrete bounds on the number of $\mc D$-compatible ideals. 
Using Theorem~4.2 of \cite{Schwede+Tucker.10} or the argument from Remark~3.4 of \cite{huneke+Watanabe.15}, the number of prime Cartier cores with respect to $\mc D$ of coheight $d$ is bounded above by $\binom{n}{d}$, where $n$ is the embedding dimension of $R$. 
\end{rmk}

\begin{thm}
\label{cts-map}
Let $(R,\mc D)$ be an $F$-finite pair, and let $\mc U_{\mc D}$ denote the $F$-pure locus of $(R,\mc D)$. Then the map $\Ccore_{\mc D}:\mc U_{\mc D} \to \mc U_{\mc D}$ is continuous under the Zariski topology.
\end{thm}
\begin{proof}
We will show that the inverse image of the closed set $ V = \bb V(J)\cap \mc U_{\mc D}$ is also closed, where $J$ is an ideal of $R$.
Let $K$ be the intersection of all Cartier cores containing $J$ which come from primes, so that
\[
K = \bigcap_{\substack{P \in \mc U_{\mc D}\\ \Ccore_{\mc D}(P)\in \mathbb V(J)}} \Ccore_{\mc D}(P).
\]
Since the set of Cartier cores with respect to $\mc D$ is closed under infinite intersection by \Cref{arbitrary-intersection-c-ideal}, $K = \Ccore_{\mc D}(K)$ is also a Cartier core. We claim that $\Ccore_{\mc D}\inv(V) = \mathbb V(K)\cap \mc U_{\mc D}$. 

Suppose $P\in \Ccore_{\mc D}\inv(V)$. Then since $P\in \mc U_{\mc D}$, we have $\Ccore_{\mc D}(P)\subseteq P$ by \Cref{prop:c-ideal-in-primary-original-ideal}. Since $\Ccore_{\mc D}(P)\in \bb V(J)$, 
we have $K\subseteq  \Ccore_{\mc D}(P)$ by construction.
Thus $K\subseteq P$ and so $\Ccore_{\mc D}\inv(V) \subseteq \bb V(K)\cap \mc U_{\mc D}$. 

Conversely, if $P\in \bb V(K)\cap \mc U_{\mc D}$, then $K\subseteq P$ and by \Cref{containment}, 
\[
J\subseteq K = \Ccore_{\mc D}(K) \subseteq \Ccore_{\mc D}(P).
\]
Thus $\bb V(K)\cap \mc U_{\mc D} \subseteq \Ccore_{\mc D}\inv(V)$.
\end{proof}


\section{Quotients of Regular Rings}
\label{sec:quo-regular-ring}
Now that we have seen some abstract properties of the Cartier core map, $\Ccore_{\mc D}$, we shift our focus to actually computing it.
In this section we give a concrete description of the Cartier core in the case when $R$ is presented as a quotient of a regular ring and $\mc D$ is the full Cartier algebra $\mc C_R$. We will then use this concrete description to show that the Cartier core commutes with adjoining a variable and with homogenization (in the case that our regular ring is a polynomial ring).

One reason to focus on this case is that regularity of $S$ forces $F_*^eS$ and $\hom_S(F_*^eS, S)$ to be well-behaved, as the following result of Kunz and result of Fedder illustrate.

\begin{thm}[{\cite[Cor.~2.7]{Kunz.69}}]
If $R$ is a Noetherian ring of prime characteristic, then $R$ is regular if and only if $F_*R$ is a flat $R$-module. 
\end{thm}

\begin{lemma}[{\cite[Lemma~1.6]{Fedder.83}}]
If $S$ is an $F$-finite regular local ring, then $\hom_S(F_*^eS, S)$ is a free rank one $F_*^eS$ module.
\end{lemma}

Further, Glassbrenner, building on work of Fedder, gives us the following description of the $R$-module structure on maps in the local case.

\begin{lemma}[{Fedder's Lemma \cite[Lemma~2.1]{Glassbrenner.96}}]
\label{qr-hom-presentation}
Let $S$ be an $F$-finite regular local ring and let $R=S/I$ for some ideal $I$. Then
\[
\hom_{R}(F_*^e R, R) \cong F_*^e\left( \frac{I^{[p^e]}:I}{I^{[p^e]}}  \right)
\]
as $R$-modules. 
\end{lemma}

This description of $\hom_R(F_*^eR, R)$ is the core of Fedder's criterion and of Glassbrenner's criterion.

\begin{prop}[{\cite[Prop.~1.7]{Fedder.83}}]
Let $(S,\mf m)$ be an $F$-finite regular local ring of prime characteristic $p$, and let $I$ be an ideal of $S$. Then $R$ is $F$-pure if and only if $(I^{[p]}:I)\not\subseteq \mf m^{[p]}$.
\end{prop}

\begin{lemma}[{\cite[Lemma~2.2]{Glassbrenner.96}}]
\label{lem:e-split-locus}
Let $(S,\mf m)$ be an $F$-finite regular local ring of prime characteristic $p$. Let $I$ be an ideal of $S$. Then the map $S/I \to F_*^e(S/I)$, where $1\mapsto F_*^ec$, splits as an $(S/I)$-module map exactly when $c\notin \mf m^{[p^e]} :(I^{[p^e]}:I)$.
\end{lemma}

Since the Cartier core of $J$ is composed precisely of the elements which cannot be split in this manner, the technique of this lemma naturally leads to the following result. For the following, we use $\olin J$ to denote the image of an ideal $J$ in a quotient ring, and similarly $\olin c$ to denote the image of an element $c$.

\begin{thm}
\label{c-ideal-quo-reg-ring}
Let $S$ be a regular $F$-finite ring, let $I\subseteq J$ be ideals of $S$, and let $R=S/I$. Fix $e\geq 1$, $c\in S$. Then there exists some $\phi \in \hom_R(F_*^eR, R)$ with $\phi(\olin c)\notin \olin{J}$ if and only if $c\notin J^{[p^e]}:(I^{[p^e]}:I)$.
In particular, 
\[
A_{e;R}(\olin J) = \olin{J^{[p^e]}:_S (I^{[p^e]}:_S I)}
\quad\textrm{ and }\quad
C_R(\olin J) = \olin{\bigcap_e J^{[p^e]}:_S (I^{[p^e]}:_S I)}.
\]
\end{thm}
\begin{proof}
The representations of $A_e$ and $\Ccore_R$ follow directly from the first statement, so it suffices to prove that $c\notin J^{[p^e]}:(I^{[p^e]}:I)$ if and only if there is some $\phi \in \hom_R(F_*^eR, R)$ with $\phi(\olin c)\notin \olin J$. 

For our fixed $e$, let $E: \hom_R(F_*^e R, R)\to R$ be the ``evaluation at $c$'' map, so that $E(\phi) = \phi(F_*^e c)$. Our goal is to show $\im(E)\subseteq \olin J$ if and only if $c\in J^{[p^e]}:(I^{[p^e]}:I)$. 
By the discussion in \Cref{sec:background}, we can view the localization of $E$ as a map $\hom_{R_P}(F_*^e(R_P), R_P)\to R_P$ so that $(\im E)_P \cong \im (E_P)$. Since localization also commutes with Frobenius and with ideal colon, we can without loss of generality assume that $(S,\mf m)$ is local.

Let $\Psi$ be a generator of $\hom_S(F_*^e S, S)$ as an $F_*^e S$ module. By \Cref{qr-hom-presentation}, the maps $\phi\in \hom_R(F_*^eR, R)$ are exactly those maps induced by something of the form $\Psi\circ F_*^e(s)$ where $s\in I^{[p^e]}:I$. Thus 
\[
\phi(F_*^e(\olin c)) = \olin{(\Psi\circ F_*^es)(F_*^ec)} = \olin{\Psi(F_*^e(sc))}
\]
and so there exists $\phi$ with $\phi(F_*^e(\olin c))\notin \olin J$ if and only if there exists $s\in I^{[p^e]}:I$ with $\Psi(F_*^e(sc))\notin J$, i.e., if and only if 
\[
\Psi\left(F_*^e\left( c(I^{[p^e]}:I)\right)\right) = (F_*^ec \cdot \Psi)(I^{[p^e]}:I)\not\subseteq J.
\]
Using \cite[Lemma~1.6]{Fedder.83}, this occurs if and only if 
\[
F_*^e(c)\notin (J F_*^eS):(F_*^e(I^{[p^e]}:I)) = F_*^e(J^{[p^e]}):F_*^e(I^{[p^e]}:I).
\]
Since $S$ is regular, the flat Frobenius commutes with colon and is injective, thus this is equivalent to
\[
c\notin J^{[p^e]}:(I^{[p^e]}:I).\qedhere
\]
\end{proof}

We will frequently move between considering $\Ccore_R(\olin J)$ in $R$ and its lift $\bigcap_{e>0}J^{[p^e]}:(I^{[p^e]}:I)$ in $S$, which we will denote as either $\wt{\Ccore}_R(J)$ or $\wt{\Ccore}_R(\olin J)$. Similarly, we will denote the lift of $A_{e;R}(\olin J)$ as $\wt A_{e;R}(J)$ or $\wt A_{e;R}(\olin J)$.

We now prove results which let us connect Cartier cores of related ideals computed in different, related rings.

\begin{lemma}
\label{qr-flat-extension-containment}
Let $S_1\to S_2$ be a flat map of regular $F$-finite rings. Consider ideals $I\subseteq J_1$ in $S_1$, and ideal $J_2$ in $S_2$ contracting to $J_1$. Let $R_1 = S_1/I$ and $R_2 = S_2/IS_2$. Then
\[
C_{R_1}(\olin{J_1})R_2 \subseteq C_{R_2}(\olin{J_2}).
\]
\end{lemma}
\begin{proof}
Finite intersections always commute with flat base change. Thus for any sequence of ideals $\{K_e\}_{e\in \N}$ and for any $n$, 
\[
\left(\bigcap_{e=1}^\infty K_e\right) S_2  \subseteq \left(\bigcap_{e=1}^n K_e\right) S_2 = \bigcap_{e=1}^n (K_eS_2)
\]
and in particular we must have $\left(\bigcap_{e=1}^\infty K_e\right) S_2 \subseteq \bigcap_{e=1}^\infty (K_eS_2)$. Colon commutes with flat base change when the ideals are finitely generated \cite[Thm.~7.4]{Matsumura.89}. Thus
\begin{align*}
\left(\bigcap_{e\geq 1}J_1^{[p^e]}:(I^{[p^e]}:I)\right) S_2
&\subseteq \bigcap_{e\geq 1} \left( (J_1S_2)^{[p^e]}:((IS_2)^{[p^e]}:IS_2) \right)
    \subseteq \bigcap_{e\geq 1} \left( J_2^{[p^e]}:((IS_2)^{[p^e]}:IS_2) \right),
\end{align*}
which by using \Cref{c-ideal-quo-reg-ring} to pass to the quotient gives
\[
C_{R_1}(\olin{J_1})R_2 \subseteq C_{R_2}(\olin{J_2}).\qedhere
\]
\end{proof}
In the case of a general flat map, even a general faithfully flat map, containment is the best we can do. 
For example, consider $S_1=k[x^p]$ and $S_2 = k[x]$ where $k$ is a perfect field. The inclusion of $S_1$ into $S_2$ is faithfully flat since it corresponds to the Frobenius on the regular ring $k[x]$. Now consider $I=J_1=\langle x^p\rangle\subset S_1$ and $J_2 = \langle x \rangle \subset S_2$. Then $R_1=S_1/I\cong K$ which is Frobenius split, so $C_{R_1}(\olin J_1) = \olin J_1$. But $R_2 = S_2/IS_2 = k[x]/\langle x^p\rangle$ is not reduced, thus cannot be Frobenius split. Since $\olin J_2$ is a prime ideal, this means $C_{R_2}(\olin J_2)=R_2$.

However, it turns out that in the case of adjoining a variable, we can get a stronger result.

\begin{prop}
\label{qr-add-var}
Let $R$ be a quotient of a regular $F$-finite ring, let $J$ be an ideal of $R$, and let $J'$ be an ideal of $R[x]$ such that $JR[x] \subseteq J' \subseteq JR[x]+\langle x \rangle$. Then
\[
C_{R}(J)R[x] = C_{R[x]}(J')
\quad \textrm{ and } \quad
C_{R[x]}(J')\cap R = C_R(J).
\]
\end{prop}
\begin{proof}
By
\Cref{containment},  
\[
C_{R[x]}(JR[x]) \subseteq C_{R[x]}(J') \subseteq C_{R[x]}(JR[x]+\langle x\rangle).
\]
Our first step will be to show
\[
C_{R[x]}(JR[x]) \supseteq C_{R[x]}(JR[x] + \langle x\rangle),
\]
which will then give us $C_{R[x]}(JR[x])=C_{R[x]}(J') = C_{R[x]}(JR[x] + \langle x \rangle)$.

To do so, note that by assumption we can write $R=S/I$ where $S$ is a regular $F$-finite ring, and so we can also write $R[x] = S[x]/IS[x]$. We use $\wt{\ }$ to denote lifting an ideal from $R$ or $R[x]$ to $S$ or $S[x]$, as appropriate. Consider $S[x]$ to be $\N$-graded by $x$. Since $\wt{JR[x]+\langle x \rangle}$, the lift of $JR[x]+\langle x\rangle$ to $S[x]$, is homogeneous, as is $IS[x]$,  our lift of the Cartier core
\[
\wt{\Ccore}_{R[x]}(J[x]+\langle x\rangle) = 
\bigcap_{e>0} \wt{JR[x]+\langle x\rangle}:(IS[x]^{[q]}:IS[x])
\]
is also homogeneous.
Consider some homogeneous $g$ in this lift of the Cartier core.
Ideal colon commutes with flat maps, and $S\to S[x]$ and the Frobenius are both flat. Thus for every $q=p^e$ we have 
\[
IS[x]^{[q]}:IS[x] = (I^{[q]}:I)S[x].
\]
Since $g\in \wt{A}_e(J)$, we must have $g(I^{[q]}:I)\subseteq (\wt{JR[x]+\langle x\rangle})^{[q]}$. However, any element of $(\wt{JR[x]+\langle x\rangle})^{[q]}$ of degree less than $q$ must be expressible in terms of elements of $\wt{JR[x]}^{[q]}$. In particular, if $q>\deg g$ then $g(I^{[q]}:I) \subseteq \wt{JR[x]}^{[q]}$. Thus for $e\gg 0$, we have 
\[
g\in \wt{JR[x]}^{[q]}:(IS[x]^{[q]}:IS[x]) = \wt{A}_{e;R[x]}(JR[x]).
\]
By \cite[Prop.~4.15]{Badilla-Cespedes.21}, since $\Ccore_{R[x]}(JR[x]) = \bigcap_{e\gg0}A_{e;R[x]}(JR[x])$, this tells us that
\[
\Ccore_{R[x]}(JR[x]+\langle x \rangle) \subseteq \Ccore_{R[x]}(JR[x])
\]
as desired.

Now we have shown $C_{R[x]}(JR[x])=C_{R[x]}(J')$, and it suffices to show  $C_R(J)R[x]=C_{R[x]}(J[x])$. 
To do so, we will show that adjoining a variable commutes with infinite intersection. 
Consider an arbitrary ideal $K = \bigcap_\alpha K_\alpha$ in $S$. 
As a set, each $K_\alpha S[x]$ is polynomials with coefficients in $K_\alpha$, and so the polynomials in $\bigcap_\alpha K_\alpha S[x]$ are those with coefficients in $K_\alpha$ for every $\alpha$, which is precisely $KS[x]$, as desired.

This lets us repeat the argument in \Cref{qr-flat-extension-containment} but with equalities, and thus 
\[
C_R(J)R[x] = C_{R[x]}(J[x]) = C_{R[x]}(J')
\]
as desired. The contraction result then follows directly from the fact that adjoining a variable is faithfully flat, so that
\[
C_R(J) = C_R(J)R[x] \cap R = C_{R[x]}(J')\cap R. \qedhere
\]
\end{proof}

If $R$ is a quotient of a polynomial ring by a homogeneous ideal, we can also look at how the Cartier core behaves under homogenization. More concretely, take $R=S/I$ for  $S=k[x_1,\ldots, x_d]$ and $I$ a homogeneous ideal of $S$, so that $R$ is $\N$-graded. If $f\in R$, we let $f^h$ denote the minimal homogenization of $f$ in $R[t]$, so that
\[
f^h = t^{\deg f} f\left(\frac{x_1}{t}, \ldots, \frac{x_n}{t}\right).
\]
If $J$ is an ideal of $R$, we define its homogenization in $R[t]$ to be $J^h = \langle f^h \: | \: f\in J\rangle$. 

For any degree-preserving lift of $f$ to $S$, there is a corresponding lift of $f^h$ to $S[t]$ so that the lift of the homogenization is the homogenization of the lift. This means we can freely consider a given homogenization to live either in $R[t]$ or in $S[t]$. Further, the ideals $\wt{(J^h)}$ and $(\wt J)^h$ are the same: $\wt{(J^h)}$ is generated by the lifts of the homogenizations of elements of $J$, and $(\wt J)^h$ is generated by homogenizations of lifts of elements of $J$.

There is also a corresponding dehomogenization map $\demog:R[t]\to R$ defined by $\demog(t)=1$, which ensures that $\demog(f^h) = f$. 

We recall the following straightforward facts about homogenization.
\begin{lemma}
Let $R$ be a quotient of a polynomial ring by a homogeneous ideal. Let $I,J$ be ideals of $R$, and $\{I_\alpha\}$ a family of ideals. Let $f$ be an element of $R$. Then the following statements all hold.
\begin{itemize}
\item $f\in I$ if and only if $f^h\in I^h$.
\item $(I:J)^h=I^h:J^h$ and $\left( \bigcap I_\alpha\right)^h = \bigcap (I_\alpha^h)$.
\item $(I^{h})^{[p^e]} = (I^{[p^e]})^h$.
\end{itemize}
\end{lemma}
\begin{proof}
For the first two bullets, see Problems~3.15 and~3.17 in \cite{Ene+Herzog.12}. For the third bullet, use Proposition~3.15 of \cite{Ene+Herzog.12} and Theorem~6.2 of \cite{Herzog+Trung.92}.
\end{proof}

Using these facts, we will prove the following lemma.
\begin{lemma}
\label{qr-homogenization}
Let $R=S/I$ where $S$ is a polynomial ring over an $F$-finite field and $I$ is a homogeneous ideal. Let $J$ be an ideal of $R$. Then
\[
\left(C_R(J)\right)^h = C_{R[t]}(J^h)
\quad\textrm{ and }\quad
C_R(J) = \demog\left( C_{R[t]}(J^h)\right)
\]
\end{lemma}
\begin{proof}
If we lift to $S[t]$ using \Cref{c-ideal-quo-reg-ring} and the above discussion on lifting and homogenization, then
\begin{align*}
\wt{\left(C_R(J)\right)^h} &= \left(\wt C_R(J)\right)^h \\
    &= \left(\bigcap_{e>0} (\wt J)^{[q]}:(I^{[q]}:I)\right)^h \\
    &= \bigcap_{e>0} \left((\wt{J}^h)^{[q]}:((I^h)^{[q]}:I^h)     \right) \\
    &=  \bigcap_{e>0} \left((\wt{J^h})^{[q]}:(I^{[q]}:I)     \right) \\
    &= \wt{C}_{R[t]}(J^h)
\end{align*}
and so contracting back to $R[t]$ via \Cref{c-ideal-quo-reg-ring},
\[
(C_R(J))^h = C_{R[y]}(J^h). 
\]

The last statement follows directly from dehomogenizing each side of the equation.
\end{proof}


\section{Stanley-Reisner Rings}
\label{sec:stanley-reisner}
A ring $R$ is a \emph{Stanley-Reisner ring} if it can be written as $R = S/I$, where $S$ is a polynomial ring and $I$ is a square-free monomial ideal. 
The following theorem gives a complete description of the Cartier core map for $\Spec R$ where $R$ is a Stanley-Reisner ring.

\begin{thm}
\label{stanley-reisner-c-map}
Let $R$ be a Stanley-Reisner ring over a field that has prime characteristic and is $F$-finite. Let $Q$ be any prime ideal. Then
\[
C_R(Q) = \sum_{\substack{P \in \Min(R) \\ P\subseteq Q}} P.
\]
In particular, the set of prime Cartier cores of $R$, i.e., the set of generic points of $F$-pure centers of $R$, is the set of sums of minimal primes.  
\end{thm}

This theorem extends some earlier results.
Aberbach and Enescu showed that the splitting prime of a Stanley-Reisner ring, which is its largest proper uniformly $F$-compatible ideal, is the sum of the minimal primes \cite[Prop~4.10]{Aberbach+Enescu.05}. For the reader's convenience, we will reprove this in our proof of \Cref{stanley-reisner-c-map}.
At the other extreme, Vassilev showed that the test ideal of a Stanley-Reisner ring, which is its smallest non-zero uniformly $F$-compatible ideal, is $\sum_{i=1}^t \bigcap_{j\neq i} P_j$ where $P_1,\ldots, P_t$ are the minimal primes of $R$ \cite[Thm.~3.7]{Vassilev.98}. 
In a related but different direction, for a specific choice of $\phi:F_*^e R\to R$, Enescu and Ilioaea showed that the $\phi$-compatible primes of $R$ are precisely the prime monomial ideals which \emph{contain} a minimal prime of $R$. They used this to give a combinatorial description the test ideal of the pair~$(R,\phi)$ \cite[Prop.~3.9, Prop.~3.10]{Enescu+Ilioaea.20}. 

Badilla-C\'espedes showed that if $P'$ is a prime monomial ideal, then $\Ccore(P')$ as well as each $A_e(P')$ is also a monomial ideal, and more explicitly that $A_e(P') = (P')^{[p^e]}+\Ccore(P')$ in this setting \cite[Lemma~4.16,Prop~4.17]{Badilla-Cespedes.21}. Meanwhile, \`Alvarez~Montaner, Boix, and Zarzuela gave a concrete description of $I^{[p^e]}:_S I$ in terms of the minimal primes of $I$, which could be used to explicitly compute the Cartier contractions for any ideal $J$ \cite[Prop~3.2]{AlvarezMontaner+etal.12}.

\begin{proof}[Proof of \Cref{stanley-reisner-c-map}]
Our proof will proceed as follows: First we will reduce to the case where every minimal prime is contained in $Q$. Then we will homogenize and trap $Q^h$ between a sum of minimal primes and the homogeneous maximal ideal, and use \Cref{containment} and the convenient form of monomial primes to get our desired equality.

Let $\wt{\ }$ denote the lift of any ideal to $S$, let $I' = \displaystyle\bigcap_{P\in \Min(R),\ P \subseteq Q} \wt P$ be the intersection of the minimal primes contained in $Q$, and let $R'=S/I'$. Then
\[
R_{Q} \cong S_{\wt Q}/I_{\wt Q} = S_{\wt Q}/I'_{\wt Q} \cong R'_{Q}
\]
and so by \Cref{localization},
\[
C_R(Q)R_{Q} = C_{R_{Q}}(Q) 
= C_{R'_{Q}}(Q) = C_{R'}(Q)R'_{Q}.
\]
Stanley-Reisner rings are $F$-pure \cite[Prop.~5.8]{Hochster+Roberts.76}, and so $C_R(Q)\subseteq Q$ by \Cref{c-ideal-in-original-ideal}, and thus when we lift back to $S$ using \Cref{c-ideal-quo-reg-ring}, we see
\[
\bigcap_{e>0} \wt Q^{[p^e]}: (I^{[p^e]}: I) = \bigcap_{e>0} \wt Q^{[p^e]}:(I'^{[p^e]}:I').
\]
Thus we can use $I'$ as our new $I$, and so we can assume $P\subseteq Q$ for all minimal primes $P$. 

Relabel the variables so that $\sum_{P\in \Min(R)} P = \langle x_1,\ldots, x_c\rangle$ and define $A = k[x_1,\ldots, x_c]/I$, so that $R = A[x_{c+1}, \ldots, x_{d}]$. 
Now we homogenize, so $Q^h \subseteq \mf m$ where $\mf m$ is the homogeneous maximal ideal in $S[t]$. 
Then \Cref{containment} tells us 
\[
C_{R[t]}\left(\sum_{P\in \Min(R)} P^h\right) 
\subseteq C_{R[t]}(Q^h) 
\subseteq C_{R[t]}(\olin{\mf m}). 
\]
Each minimal prime $P$ of $R$ remains a minimal prime of $R[t]$ after homogenizing, so \Cref{c-minl-prime} says $C_{R[t]}(P^h) = P^h$, and \Cref{prop:arbitrary-sum-c-ideals} then says that their sum is also preserved by the Cartier core map.
Applying \Cref{qr-add-var} to $\olin{\mf m}$, we get
\[
\langle x_1,\ldots, x_c\rangle R[t] 
= C_{A}(P_1+\cdots+P_t)A[x_{c+1},\ldots, x_d,t] 
= C_{R[t]}(\olin{\mf m}).
\]

Thus 
\[
\langle x_1,\ldots, x_c\rangle 
= \Ccore_{R[t]}\left(\sum_{P\in \Min(R)} P^h \right) \subseteq C_{R[t]}(Q^h) 
\subseteq C_{R[t]}(\olin{\mf m}) 
= \langle x_1,\ldots, x_c\rangle
\]
and by \Cref{qr-homogenization},
\[
(C_{R}(Q))^h = C_{R[t]}(Q^h) = \langle x_1,\ldots, x_c\rangle.
\]
Dehomogenizing the homogenization always gives back the original ideal, and so
\[
C_R(Q) = \langle x_1,\ldots, x_c\rangle = \sum_{P\in \Min(R)} P.
\]

For the last statement of the theorem, note that since each minimal prime of $R$ corresponds to an ideal of $S$ which is generated by variables, any sum of minimal primes is also prime, and thus is fixed by the Cartier core map.
\end{proof}

Since taking the Cartier core commutes with intersection, \Cref{stanley-reisner-c-map} immediately gives a formula for the Cartier core of any radical ideal in terms of the Cartier cores of its minimal primes. The following corollary instead gives a formula for the Cartier core of an arbitrary ideal which is more analogous to the previous one. 

\begin{cor}
\label{stanley-reisner-general-c-map}
Let $R$ be a Stanley-Reisner ring over a field that has prime characteristic and is $F$-finite. Let $J$ be any ideal. Then
\[
\Ccore_R(J) = \sum_{\substack{\mc Q\subset \Min(R) \\ \big({\bigcap\limits_{P\in \mc Q}} P\big) \subset J  }} 
    \left( \bigcap_{P\in \mc Q} P \right).
\]
\end{cor}
\begin{proof}

First, we will show our desired formula gives an ideal contained in $\Ccore_R(J)$. The restriction on the $P$'s appearing ensures that the resulting ideal is contained in $J$. Further, each $P$ appearing is minimal, so $P= \Ccore_R(P)$ by \Cref{c-minl-prime}. Using Propositions~\ref{arbitrary-intersection-c-ideal} and~\ref{prop:arbitrary-sum-c-ideals} (our intersection and sum results) to apply $\Ccore_R$ to the formula and then using \Cref{containment} to preserve the containment gives
\[
\sum_{\substack{\mc Q\subset \Min(R) \\ \big({\bigcap\limits_{P\in \mc Q}} P\big) \subset J  }} 
        \left( \bigcap_{P\in \mc Q} P \right)
    =\Ccore_R\left( \sum_{\substack{\mc Q\subset \Min(R) \\ \big({\bigcap\limits_{P\in \mc Q}} P\big) \subset J  }} 
        \left( \bigcap_{P\in \mc Q} P \right)\right)
    \subset \Ccore_R(J).
\]

To show equality, we will show that summing over a specific smaller subset in fact already yields $\Ccore_R(J)$, and so the larger sum above must yield $\Ccore_R(J)$ as well.
Since $R$ is Frobenius split, $\Ccore_R(J)$ is radical by \Cref{c-ideal-radical}, so we can write $\Ccore_R(J) = \bigcap_{i=1}^n Q_i$ as the intersection of its minimal primes. 
By the same argument as in the proof of \Cref{cts-map}, applying \Cref{proper-prime}, \Cref{c-ideal-in-original-ideal}, and \Cref{prop:c-ideal-stable}
shows that $\Ccore_R(Q_i)=Q_i$ for each $i$.

Now since the Cartier core commutes with intersection, we use \Cref{stanley-reisner-c-map} to see
\[
\Ccore_R(J) = \bigcap_i \Ccore_R(Q_i) = \bigcap_i \left( \sum_{\substack{P\in \Min(R)\\P\subset Q_i}} P\right).
\]
Writing $R=S/I$ as a quotient of a polynomial ring and lifting back up to $S$, this says that the lift of $\Ccore_R(J)$ is an intersection of sums of monomial ideals. Sum and intersection of monomial ideals commute \cite[Problem~1.17]{Ene+Herzog.12},
and so passing back to the quotient gives
\[
\Ccore_R(J) = \sum_{\substack{P_1,\ldots, P_n \in \Min(R) \\ P_i \subset Q_i }} \left(\bigcap_{i=1}^n P_i\right).
\]
For each possibility for $P_1,\ldots, P_n$ in the above sum, we have $\bigcap_{i=1}^n P_i 
\subset \Ccore_R(J) \subset J$,
and so this sum is a subset of our desired formula, which thus must be equal to $\Ccore_R(J)$ as well.
\end{proof}

\bibliographystyle{abbrvnat}
\bibliography{cartier-core}

\end{document}